\documentclass[11pt]{article}
\usepackage{amsthm, amsmath, amssymb, amsfonts, url, booktabs, tikz, setspace, fancyhdr, bm}
\usepackage{geometry}
\usepackage{hyperref, enumerate}
\usepackage[shortlabels]{enumitem}
\usepackage[babel]{microtype}
\usepackage[english]{babel}
\usepackage[capitalise]{cleveref}
\usepackage{comment}
\usepackage{bbm}
\usepackage{csquotes}
\usepackage{mathabx}
\usepackage{tikz}
\usepackage{graphicx}
\usepackage{float}
\usepackage{xcolor}
\usepackage{soul}
\usepackage{mathtools}
\usetikzlibrary{positioning, arrows.meta, shapes.geometric}

\counterwithin{figure}{section}

\newtheorem{theorem}{Theorem}[section]
\newtheorem{prop}[theorem]{Proposition}
\newtheorem{conj}[theorem]{Conjecture}
\newtheorem{lemma}[theorem]{Lemma}
\newtheorem{cor}[theorem]{Corollary}
\newtheorem{claim}[theorem]{Claim}

\newtheorem{construction}{Construction}

\usetikzlibrary{decorations.pathmorphing}

\theoremstyle{definition}

\newtheorem*{defn-non}{Definition}

\newlist{Case}{enumerate}{2}
\setlist[Case, 1]{%
    label           =   {\bfseries Case \arabic*.},
    labelindent=1em ,labelwidth=1.3cm, labelsep*=1em, leftmargin =!
}
\setlist[Case, 2]{%
    label           =   {\bfseries Subcase \arabic{Casei}.\arabic*.},
    labelindent=-1em ,labelwidth=1.3cm, labelsep*=1em, leftmargin =!
}

\usepackage{todonotes}

\newcommand{\ex}{\mathrm{ex}}

\title{Triple systems with bounded matching number: some constructions and exact Tur\'{a}n number} 
\author{
Nannan Chen\thanks{School of Mathematics, Shangdong University, Jinan, China, and Extremal Combinatorics and Probability Group (ECOPRO), Institute for Basic Science (IBS), Daejeon, South Korea. 
Emails: nannanchen@sdu.edu.cn, yzq\_sdu\_edu@163.com.}
\and
Miao Liu\thanks{Research Center for Mathematics and Interdisciplinary Sciences, Shandong University, Qingdao, China, and Extremal Combinatorics and Probability Group (ECOPRO), Institute for Basic Science (IBS), Daejeon, South Korea. Email: liumiao10300403@163.com.}
\and
Yuzhen Qi\footnotemark[1]
\and
Caihong Yang\thanks{School of Mathematics and Statistics, Fuzhou University, Fujian,
China, and Extremal Combinatorics and Probability Group (ECOPRO), Institute for Basic Science (IBS), Daejeon, South Korea.  Email: chyang.fzu@gmail.com.}
}
\begin{document}
\date{}
\maketitle
\begin{abstract}    
We study the Tur\'{a}n numbers of $3$-graphs avoiding $3$-graphs $F$ and $M_{s+1}^3$, a matching of size $s+1$. We disprove a conjecture of Gerbner, Tompkins, and Zhou [European Journal of Combinatorics, 2025, 127:104155] on $\ex(n,\{F,M^3_{s+1}\})$ for $3$-graph $F$ with $\chi(F)=2$ by constructing infinitely many counterexamples. For this family, we determine the asymptotic Tur\'{a}n number via edge-colored Tur\'{a}n problem. In addition, for the $3$-graph $F_{3,2}$ with edge set $\{123,145,245,345\}$, we determine the exact value of $\ex(n,\{F_{3,2}, M_{s+1}^3\})$ for every integers $s$ and all $n \ge 12s^2$.

\end{abstract}

\section{Introduction}\label{sec:introduction}
    A long line of research in extremal combinatorics, initiated by Tur\'{a}n in the 1940s, is to  determine the maximum size of an $r$-graph on $n$ vertices that contains no fixed $r$-graph $F$; that is, the \textit{Tur\'{a}n number} $\ex(n,F)$. The \textit{Tur\'{a}n density} of $F$ is $\pi( F):=\lim_{n\to \infty}\ex(n,F)/{n\choose r}$. For graphs ($r=2$), the Tur\'{a}n numbers are quite well understood when the forbidden subgraph $F$ is non-bipartite. However, there are very few results even for specific hypergraphs, let alone general results for families of hypergraphs; see \cite{Keevash11} for more details. 

    Given an $r$-graph $F$, the \textit{matching number} $\nu(F)$ is defined as the maximum number of pairwise disjoint edges in $F$. Let $M_{s+1}^r$ denote the matching of size $s+1$ with each edge being an $r$-set. In a foundational result, Erd\H{o}s and Gallai \cite{erdos_gallai1959} in 1959 determined the Tur\'{a}n number of $M_{s+1}^2$. Erd\H{o}s later extended his attention to hypergraphs, and proposed the following conjecture for general $r \geq 3$ in 1965. There are two natural families $\binom{[r(s+1)-1]}{r}$ and $\binom{[n]}{r}\setminus \binom{[n-s]}{r}$ with matching number at most $s$. Erd\H{o}s conjectured that for any $n,s,r$, one of these families achieves the maximum value. 

 \begin{conj} [Erd\H{o}s Matching Conjecture, \cite{erdos1965}] \label{conj:matching}
     Let $n, s, r$ be positive integers such that $r \geq 3$ and $n \geq r(s+1)-1$. Then 
     $$\ex(n, M_{s+1}^r)\leq \max \left\{\binom{r(s+1)-1}{r},\binom{n}{r}-\binom{n-s}{r}\right\}.$$
 \end{conj}

\noindent The case $s=1$ is the classical Erd\H{o}s-Ko-Rado theorem \cite{EKR1961}. For $r=3$, the conjecture was settled  through a series of works \cite{frankl2017, frankl_rodl_rucinski2012, luczak_mieczkowska2014}. For general $r$, Erd\H{o}s \cite{erdos1965} himself proved the conjecture for $n > n_0(r, s)$. For more developments and partial results on this conjecture, we refer the reader to \cite{frankl2017proof,frankl2022erdHos,huang2012size}.

Motivated by \cref{conj:matching}, Tur\'an-type problems for $r$-graphs under matching constraints have been extensively studied. In the graph case ($r=2$), Alon and Frankl \cite{alon2024turan} determined the exact value $\ex(n,\{K_t,M^2_{s+1}\})$ for $n\ge 2s+1$ and $\ex(n,\{F,M^2_{s+1}\})$ for $n\ge n_0(s)$ when $F$ is a color-critical graph and $\chi(F)\geq 3$. Subsequently, Gerbner \cite{gerbner2024turan} extended the latter result to arbitrary graphs $F$. For $r\geq 3$, Gerbner, Tompkins and Zhou \cite{gerbner2025} investigated $\ex(n,\{F,M^r_{s+1}\})$ for all $r$-graphs $F$ with $\chi(F)\geq 3$. Moreover, in the $2$-chromatic case, they established general lower and upper bounds for $\ex\left(n,\left\{F, M_{s+1}^r\right\}\right)$ and resolved several special cases, as summarized below. 

Before stating the theorem, we record some notation. Let $F$ be an $r$-uniform hypergraph with $\chi(F)=2$. Write $p(F)$ for the minimum number of red vertices over all proper red-blue colorings of $F$. A \emph{strong red–blue coloring} of $F$ is a proper red–blue coloring in which every edge contains exactly one red vertex. Let $q(F)$ denote the minimum number of red vertices over all strong red–blue colorings of $F$, and set $q(F)=\infty$ if no such coloring exists. Clearly, $p(F)\le q(F)$.

\begin{theorem}[\cite{gerbner2025}]\label{thm:gerbner}
Let $F$ be an $n$-vertex $r$-graph with  $\chi(F)=2$ and $n \geq(2 s+1) r-s$.  
Then, 
\begin{itemize}
    \item if $s<\min \{p(F), r\}$, then $\ex\left(n,\left\{F, M_{s+1}^r\right\}\right)=\sum_{i=1}^s\binom{s}{i}\binom{n-s}{r-i}$. 
    \item If $s<q(F)$, then $\ex\left(n,\left\{F, M_{s+1}^r\right\}\right)=s\binom{n-s}{r-1}+O\left(n^{r-2}\right)$. 
\end{itemize}
\end{theorem}

For a $3$-graph $F$ with $\chi(F)=2$ and $s\ge q(F)$, Gerbner, Tompkins, and Zhou~\cite{gerbner2025} introduced \cref{cons:conj} below, which is $F$-free and has matching number at most $s$, and on this basis posed Conjecture \ref{conj1}. We first recall the construction underlying the conjecture and fix notation. For every $v\in V(F)$, the \emph{link graph} of $v$ in $F$ is the graph $L_F(v)$ on $V(F)\setminus\{v\}$ whose edges are precisely those pairs $\{x,y\}$ with $vxy\in E(F)$. The \emph{Tur\'an graph} $T(n,t)$ is the complete $t$-partite graph on $[n]$ whose part sizes differ by at most $1$; by Tur\'an's theorem it is (up to isomorphism) the unique extremal $K_{t+1}$-free graph on $n$ vertices. 
We write $t(n,t):=e(T(n,t))$ for its number of edges. 

\begin{construction}[\cite{gerbner2025}]\label{cons:conj}
Let $F$ be a $3$-graph with $\chi(F)=2$ and $q(F)\le s$. Order the vertices of $F$ as $1,2,\dots,|V(F)|$ so that
$$\chi\!\big(L_F(1)\big)\ \ge\ \chi\!\big(L_F(2)\big)\ \ge\ \cdots\ \ge\ \chi\!\big(L_F(|V(F)|)\big).$$
For each $i\in[q(F)]$, set $\ell_i:=\chi(L_F(i))$ and define a $3$-graph $H_i$ as follows. Let $A_1,A_2,B$ be pairwise disjoint with $|A_1|=i-1$, $|A_2|=s-(i-1)$ and $|B|=n-s$. Place a copy of the Tur\'an graph $T(n-s,\ell_i-1)$ in $B$. Finally, include all hyperedges of the following two types:
\begin{enumerate}
\item[$(1)$] edges containing one vertex from $A_1$ together with two vertices from $B$; \label{cons:type1}
\item[$(2)$] edges containing one vertex from $A_2$ together with an edge of $T(n-s,\ell_i-1)$ lying in $B$. \label{cons:type2}
\end{enumerate}
\end{construction}

\noindent
We record the basic properties of $H_i$. First, every edge of $H_i$ meets $A_1\cup A_2$, so $\nu(H_i)\le |A_1|+|A_2|=s$. Next, when $\ell_i>2$, the link graph of any vertex $a\in A_2$ is exactly $T(n-s,\ell_i-1)$ and hence $(\ell_i-1)$-partite; the link of any $b\in B$ is bipartite (it consists of all pairs joining $A_1$ to $B$ together with those joining $A_2$ to the parts of $B$ distinct from the part of $b$). In contrast, for any $a\in A_1$, its link graph induced by $B$ is complete (see edges of type (1)), and thus has chromatic number $|B|$. Consequently, in $H_i$ there are at most $i-1$ vertices (namely, those in $A_1$) whose link graph has chromatic number at least $\ell_i$, while every other vertex has link chromatic number at most $\ell_i-1$. Since $F$ contains at least $i$ vertices whose link chromatic numbers are $\ge \ell_i$ by construction of the ordering, $H_i$ is $F$-free. If $\ell_i=2$, then $T(n-s,1)$ has no edges, so type (2) contributes no hyperedges inside $A_2\cup B$. In this case, coloring all vertices of $A_1$ red and all vertices of $A_2\cup B$ blue yields a strong red–blue coloring with fewer than $q(F)$ red vertices. This again forbids a copy of $F$ in $H_i$. Thus, for each $i\in[q(F)]$, the hypergraph $H_i$ is $F$-free and satisfies $\nu(H_i)\le s$. These observations naturally lead to the following conjecture. 

\begin{conj}[\cite{gerbner2025}]\label{conj1}
Let $F$ be a $3$-graph with $\chi(F)=2$. If $q(F)\le s$ and $n$ is sufficiently large, then
\[
\ex\bigl(n,\{F,M_{s+1}^3\}\bigr)
=\max\{\,|E(H_i)|: i\in[q(F)]\,\}+o(n^2).
\]
\end{conj}

\subsection{Counterexamples and edge-colored Tur\'an}  

In this paper, we construct an infinite family of $3$-graphs that refutes Conjecture~\ref{conj1}. Furthermore, we determine the correct asymptotic Tur\'{a}n numbers for this family.

\begin{construction}\label{cons:conj2}
Let $t\ge 3$ and let $K_t$ be the complete graph on vertex set $[t]$. Partition $E(K_t)$ into $t-1$ stars
\[
\mathcal{P}:=\{S_1,S_2,\dots,S_{t-1}\},
\]
where for each $i\in[t-1]$, the star $S_i$ consists of all edges $ij$ with $j>i$. Define the bipartite $3$-graph $F_{\mathcal{P},t}$ on the vertex set $A\cup [t]$, where $A=\{a_1,\dots,a_{t-1}\}$ is disjoint from $[t]$, by
\[
E(F_{\mathcal{P},t})=\bigl\{\,a_iij:\ a_i\in A\ \text{and}\ ij\in E(S_i)\ \text{for}\ i\in[t-1]\,\bigr\}.
\]
\end{construction}

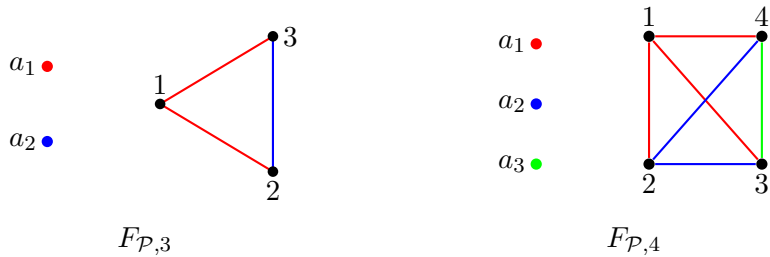
\begin{figure}[H]
\centering
\begin{tikzpicture}
\label{figure-star-partition}

\node[circle,fill=red,inner sep=1.5pt] (redpoint) at (0,1) {};
\node[circle,fill=blue,inner sep=1.5pt] (bluepoint) at (0,0) {};

\coordinate (A) at (1.5,0.5);
\coordinate (B) at (3,-0.4);
\coordinate (C) at (3,1.4);

\draw[red,thick] (A) -- (B);
\draw[red,thick] (A) -- (C);
\draw[blue,thick] (B) -- (C);

\fill[black] (A) circle (2pt);
\fill[black] (B) circle (2pt);
\fill[black] (C) circle (2pt);

\node[left] at (redpoint) {$a_1$};
\node[left] at (bluepoint) {$a_2$};
\node[above] at (A) {$1$};
\node[below] at (B) {$2$};
\node[right] at (C) {$3$};

\node[below] at (1.3,-1) {$F_{\mathcal{P},3}$};

\node[circle,fill=blue,inner sep=1.5pt] (bluepoint) at (6.5,0.5) {};
\node[circle,fill=green,inner sep=1.5pt] (greenpoint) at (6.5,-0.3) {};
\node[circle,fill=red,inner sep=1.5pt] (redpoint) at (6.5,1.3) {};

\node[circle,fill=black,inner sep=1.5pt] (D) at (9.5,1.4) {};
\node[circle,fill=black,inner sep=1.5pt] (C) at (9.5,-0.3) {};
\node[circle,fill=black,inner sep=1.5pt] (A) at (8,1.4) {};
\node[circle,fill=black,inner sep=1.5pt] (B) at (8,-0.3) {};


\draw[red,thick] (A) -- (B);
\draw[red,thick] (A) -- (C);
\draw[red,thick] (A) -- (D);
\draw[blue,thick] (B) -- (C);
\draw[blue,thick] (B) -- (D);
\draw[green,thick] (C) -- (D);

\node[left] at (redpoint) {$a_1$};
\node[left] at (bluepoint) {$a_2$};
\node[left] at (greenpoint) {$a_3$};
\node[above] at (A) {$1$};
\node[below] at (B) {$2$};
\node[below] at (C) {$3$};
\node[above] at (D) {$4$};

\node[below] at (7.8,-1) {$F_{\mathcal{P},4}$}; 
\end{tikzpicture}
\caption{Bipartite $3$-graphs $F_{\mathcal{P}, 3}$ and $ F_{\mathcal{P}, 4}$. The red, blue and green edges belong to the link graph of red, blue and green vertices, respectively. }
\end{figure}
Note that $F_{\mathcal{P},3}=F_5$ when $t=3$, see Figure~\ref{figure-star-partition}. While the exact value of $\ex(n,\{F_5,M_{s+1}^3\}\bigr)$ was determined in~\cite{wang2025}, the cases $t\ge 4$ remain open. Our first result determines the asymptotic behavior of $\ex\bigl(n,\{F_{\mathcal{P},t},M_{s+1}^3\}\bigr)$ for every $t\ge 4$, thereby providing infinite counterexamples to Conjecture \ref{conj1}.  

\begin{theorem}\label{thm:counter}
For each $s\geq 1$, $t\geq 4$ and sufficiently large $n$,  
$$\ex(n, \{F_{\mathcal{P}, t},M_{s+1}^3\}) =
\begin{cases}
s\binom{n-s}{2} + O(n), & \text{if } s \leq t-2,\\
s\cdot t(n-s, t-1) + O(n), & \text{if } s \geq t-1.
\end{cases}$$
\end{theorem}

We prove this theorem by invoking an edge-colored Tur\'an-type result stated as follows.

\begin{theorem}\label{thm:star colored K_r-free}
    For $s\ge 1$, $r \geq 3$ and $n> r^3$, let $G_1,G_2,\cdots,G_s\subseteq \binom{[n]}{2}$ be $(r-1)$-star edge-colored $K_r$-free. Then $$\sum_{i=1}^{s}e(G_i)\leq
    \begin{cases}
        s\binom{n}{2}, & \text{if } s \leq r-2,\\ 
        s \cdot t(n,r-1) + sn, & \text{if } s \geq r-1.
    \end{cases}
    $$ 
\end{theorem}

We remark that Theorem \ref{thm:star colored K_r-free} is tight up to the additive term $sn$. Moreover, in \Cref{sec:concluding remarks}, we propose a conjecture that this term can be removed, which would yield the exact extremal value in the second regime.

\begin{cor}\label{cor:conterexample}
There exists an infinite family of $3$-graphs that violates Conjecture \ref{conj1}. 
\end{cor}

\subsection{A positive case for Conjecture \ref{conj1}} 
The \textit{full star} $J_t$ is defined as the $3$-graph on the $t+1$ vertices consisting of a center vertex $v_0$ and other $t$ vertices, where the edge set consists of all triples containing $v_0$. The $3$-graph $K_4^{3-}$ is obtained by removing one edge from the complete $3$-graph $K_4^3$. Observe that $J_3$ is isomorphic to $K_4^{3-}$. Known results on the Tur\'an number of $J_t$ are rare. In fact, the best known lower bound construction arises from the daisy graph \cite{bollobas2011daisies}. In this paper, we determine the asymptotic value  of $\ex(n,\{J_t,M_{s+1}^3\})$, which serves as a positive example for Conjecture \ref{conj1}. 

\begin{theorem}\label{thm:J_t}
    For each $s\geq 1$, $t\geq 3$ and sufficiently large $n$, 
    $$\ex(n,\{J_t,M_{s+1}^3\})= s\cdot t(n-s, t-1)+O(n).$$ 
    In particular, for $t=3$, $s=1,2$ and  $n \geq 150$, $$
   \ex(n, \{K_4^{3-}, M_{s+1}^3\}) =
   \begin{cases}
      2 \lfloor \frac{(n - 2)^2}{4} \rfloor + 1, &\text{ if } s=2 \text{ and } n \text{ is odd}, \\
      s \lfloor \frac{(n - s)^2}{4} \rfloor, &\text{ otherwise}.
   \end{cases}
   $$ 
\end{theorem}

\subsection{The Tur\'{a}n number of $\{F_{3,2},M_{s+1}^3\}$}

Lastly, we study the $3$-graphs  $F_{3,2}=([5],\{123,145,245,345\})$, a fundamental structure in hypergraph Tur\'an theory. In graphs, triangle-freeness is equivalent to the independence of all vertex neighborhoods. Analogously, in a $3$-graph, the neighborhood of every pair is an independent set if and only if the hypergraph is $F_{3,2}$-free. The extremal problem for $F_{3,2}$ was initially motivated by a Ramsey–Tur\'an hypergraph problem introduced by Erd\H{o}s and S\'os \cite{erdHos1979problems}. The notation $F_{3,2}$ was later adopted by R\"odl and Mobayi \cite{MR02}. The Tur\'an density of $F_{3,2}$ was first determined by F\"uredi, Pikhurko, and Simonovits \cite{furedi2003turan}, who also established its exact Tur\'an number subsequently \cite{furedi2005triple}. 

We extend this line of research by determining the exact Tur\'an number $\ex(n, 
\{F_{3,2}, M_{s+1}^3\})$ for all $n \ge 12s^2$ and providing a characterization of the extremal structure. To this end, we first introduce the following construction.  
\begin{construction}\label{construction2}
  Let $A$ and $B$ be two disjoint vertex sets with $|A| = s$ and $|B| = n - s$. Define a $3$-graph $\mathcal{H}(n,s)$ on the vertex set $A \cup B$, where the edge set is given by
    $$E(\mathcal{H}(n,s)) := \left\{ ab_1b_2: a \in A,\; b_1, b_2 \in B,\; b_1 \ne b_2 \right\}.$$  
\end{construction}
\noindent Clearly, $e(\mathcal{H}(n,s)) = s\binom{n-s}{2}$, and $\mathcal{H}(n,s)$ is $F_{3,2}$-free and $\nu(\mathcal{H}(n,s)) \leq s$. Therefore, we have the lower bound $\ex(n,\{F_{3,2},M_{s+1}^3\})\ge s\binom{n-s}{2}$. Further, we establish that this construction is optimal for all $n \ge 12s^2$, thereby determining the exact Tur\'an number in this range.

\begin{theorem}\label{thm:f32}
   Let \( n \geq 12 s^2 \). Then
   $$\ex(n, \{F_{3,2}, M_{s+1}^3\}) = s \binom{n - s}{2},$$
   and $\mathcal{H}(n,s)$ is the unique extremal construction.
\end{theorem} 

One can readily verify that $q(F_{3,2})=\infty$ and we determine the exact value of $\ex(n,\{F_{3,2},M^r_{s+1}\})$. This leads to a natural question: for every $3$-graph $F$ satisfying $q(F)=\infty$, is it true that $\ex(n,\{F,M_{s+1}^3\})$\\$=s\cdot \binom{n-s}{2}$? However, as we will show in \cref{sec:f32}, the exact Tur\'an numbers of such $3$-graphs $F$ are not always the same.

\vspace{2mm}

\textbf{Organization.} The rest of this paper is structured as follows. In \Cref{sec2}, we introduce necessary notation and preliminary lemmas. In Sections \ref{sec:counterexample} and \ref{sec:examples}, we present the proofs of \Cref{thm:counter} and \Cref{thm:J_t} respectively, which provide several counterexamples and supporting examples to Conjecture \ref{conj1}. \Cref{sec:f32} contains the proof of \Cref{thm:f32}. Concluding remarks and future directions are discussed in \Cref{sec:concluding remarks}.

\section{Preliminary}\label{sec2}
We define some necessary notation. Let $n,m\in \mathbb{N}$ and $m\leq n$, write $[n]=\{1,2,\ldots,n\}$ and $[m,n]=\{m,m+1,\ldots,n\}$. For any $X \subseteq [n]$, we use $\overline{X}$ to denote the complement $[n] \setminus X$.  Let $H$ be a 3-graph and $S \subseteq V(H)$, define $H[S]$ as the subgraph of $H$ induced by $S$.  Given a vertex $u \in V(H)$, set 
$$L_{H}(u,S)=\left\{v w \in\binom{S}{2}: u v w \in E(H)\right\}$$
and $d_{H}(u, S)=\left|L_{H}(u, S)\right|$.
To simplify notation, we use $L_{H}(u)$  and $d_{H}(u)$ instead of $L_{H}(u,V(H))$  and $d_{H}(u,V(H))$. 
For $u, v\in V(H)$, let $N_{H}(u, v)$ denote the edges in $H$ containing both $u$ and $v$, i.e., 
\begin{center}
 $N_{H}(u, v)=\{w \in V(H): u v w \in E(H)\}$ and $d_{H}(u, v)=\left|N_{H}(u, v)\right|$.   
\end{center}
 When the context is clear, we often omit the subscript $H$. For $S, S^{\prime} \subseteq V(H)$ and $S \cap S^{\prime}=\emptyset$, let $H\left[S, S^{\prime}\right]$ denote the subgraph on the vertex set $S \cup S^{\prime}$ with the edge set consisting of all $e \in E(H)$ intersecting both $S$ and $S^{\prime}$.  

We say that a subset $S\subseteq V(H)$ is a \textit{weakly independent set} of $H$ if $S$ does not contain any edge of $H$. Intuitively, the next lemma shows that if $H$ has bounded matching number, then the number of high degree vertices is small, yet they almost cover all edges of $H$.
\begin{lemma}\label{fact}
    Let $H$ be a $3$-graph on the vertex set $[n]$ with $\nu(H)\leq s$. Define 
    \begin{equation}\label{eq:A and B}
        A:=\{v\in V(H) \colon d_{H}(v) \geq 3sn+1\} \text{ and } B:= [n] \setminus A. 
    \end{equation}
Then the following statements hold: 
  \begin{enumerate}[$(1)$]
        \item $e(H[B]) \leq 9s^2n$;
        \item $|A| \leq s$. In particular, if $|A| = s$, then $B$ is a weakly independent set.
    \end{enumerate}
\end{lemma}
\begin{proof}
    (1)  Since $\nu(H[B])\leq \nu(H) \leq s$, the subgraph $H[B]$ admits a vertex cover $X \subseteq B$ with $|X|\le 3s$. In other words, every edge of $H[B]$ must intersect $X$. By the definition of $B$, for every $v\in B$, $d_{H[B]}(v)\le d_{H}(v) \le 3sn$. 
    Consequently, 
    $$e(H[B]) \leq \sum_{v\in X}d_{H[B]}(v) \le  3s \cdot 3sn=9s^2n. $$
    
    (2) Suppose, for contradiction, that $|A| > s$. Let $A' = \{v_1, v_2, \ldots, v_{s+1}\}$  be a subset of $A$. For each $v_i \in A'$, $i\in [s+1]$, we iteratively select $u_{i_1}u_{i_2} \in E(L_{H}(v_i))$ such that
\begin{itemize}
    \item $\{u_{i_1}, u_{i_2}\} \cap A'\setminus \{v_i\} = \emptyset$, and
    \item Neither $u_{i_1}$ nor $u_{i_2}$ is selected in any of the previous steps.
\end{itemize}
At each step $i$, there are at most $(s+2(i-1))n$ edges in $L_{H}(v_i)$ that contain a vertex in $A'$ or in the earlier selections.  Since each $v_i\in A$ satisfies 
$$d_{H}(v_i) \geq 3sn+1> (s+2(i-1))n,$$ 
there always exists a valid choice of ${u_{i_1}u_{i_2}}$. Proceeding in this way, we obtain $s+1$ pairwise disjoint edges of $H$, forming a matching of size $s+1$. This contradicts to the assumption $\nu(H)\le s$. Hence we must have $|A|\le s$. The similar argument could show that if $|A| = s$, then $B$ is a weakly independent set.
\end{proof}

\section{Edge-colored Tur\'an problem}\label{sec:edge-colored turan}

Let $G_1, G_2, \cdots,G_s$ be simple graphs on the same vertex set $V$. Define an \textit{edge-colored multiple graph} $G = G_1+G_2+·\cdots+G_s$ as a multiple graph on the vertex set $V$ with edges in $E(G_i)$ having color $i$ for $i = 1,2,\cdots,s$. 
The \emph{multiplicity} of an edge $e
\in \binom{V}{2}$ is written $\omega(e)$. 
Given two sets $S,T\subseteq V(G)$, we use $e_G(S,T)$ to denote the number of edges between $S$ and $T$ in $G$. When $S = \{v\}$, we simply write $e_G(v, T)$ for $e_G(\{v\}, T)$. 
In particular, for $T = V(G)$, we have $e_G(\{v\}, V(G)) = d_G(v)$. 
Further, define the \textit{minimum degree} 
$\delta(G):=\min\{d_G(v):v\in V(G)\}$.  When the context is clear, we often omit the subscript $G$.

An edge-colored $K_t$ is called \textit{$k$-bipartite} if it uses exactly $k$ colors and the subgraph induced by each color class is bipartite. In particular, if the subgraph induced by each color class is a star, then it is called \textit{$k$-star edge-colored $K_t$}. We say $G_1, G_2, \cdots , G_s$ is $k$-star edge-colored $K_t$-free if the edge-colored multiple graph $G = G_1+G_2+·\cdots+G_s$ is $k$-star edge-colored $K_t$-free. 

\begin{theorem}[\cite{wang2025}]\label{thm:2-colored-triangle}
    Let $s\geq 2$ and $G_1, \dots, G_s \subseteq \binom{[n]}{2}$ be $2$-star edge-colored $K_3$-free. Then $$\sum\limits_{i=1}^s e(G_i) \leq s\cdot \left\lfloor\frac{n^2}{4}\right\rfloor.$$
\end{theorem} 

This result can be viewed as a colored analogue of Mantel’s theorem, and it is extended to larger cliques in Theorem \ref{thm:star colored K_r-free}. We now proceed to the proof of Theorem \ref{thm:star colored K_r-free}.
\begin{proof}[Proof of \cref{thm:star colored K_r-free}]
For $s \leq r-2$, it's trivial that $\sum_{i=1}^{s}e(G_i)\leq s\binom{n}{2}$. We now treat the case $s\geq r-1$ by induction on $s$. 
For the base case $s = r - 1$, we assume that the claim holds, i.e., 
$$\sum_{i=1}^{r-1}e(G_i)\le (r-1) \cdot t(n,r-1)+(r-1)n.$$ 
Assume that the statement holds for all integers less than $s$, in particular, for every subset $S\subseteq [s]$ with size $s-1$, we have 
$$\sum_{i\in S}e(G_i)\leq (s-1)\cdot t(n,r-1) + (s-1)n.$$
There are exactly $s$ such subsets $S$, and each graph $G_i$ appears in exactly $s-1$ of them. Therefore, summing over all subsets $S$ of size $s-1$, we obtain 
$$\sum_{S\in \binom{[s]}{s-1}}\sum_{i\in S}e(G_i)=(s-1)\sum_{i=1}^s e(G_i).$$
On the other hand, by the induction hypothesis, we get 
$$\sum_{S\in \binom{[s]}{s-1}}\sum_{i\in S}e(G_i)\le s\cdot \left((s-1)\cdot t(n,r-1) + (s-1)n\right).$$ 
Combining the two expressions yields $\sum_{i=1}^{s}e(G_i)\le s \cdot t(n,r-1) + sn$. It suffices to consider the base case $s=r-1$. 

Suppose that $G_1,\cdots,G_{r-1}$ is $(r-1)$-star edge-colored $K_r$-free with maximal total number of edges. Let $G = G_1+G_2+\cdots+G_{r-1}$ be the edge-colored multiple graph. Suppose to the contrary that $e(G) > (r-1) \cdot t(n,r-1)+(r-1)n$. 
Assume that $\delta(G)\geq (r-1)(n-\lceil \frac{n}{r-1}\rceil) + r-1$. Otherwise, let $G = G(n)\supseteq G(n-1)\supseteq\cdots$ be a sequence of induced multigraphs of $G$ as follows. If $G(m+1)$ has a vertex $v_i$ of degree less than $(r-1)(m+1-\lceil \frac{m+1}{r-1}\rceil) + r-1$, then let $G(m):=G(m+1)-v_i$. If not, we terminate our sequence. 

Set $$f(m)=e(G(m))-(r-1) \cdot t(m,r-1)-(r-1)m,$$ 
we have $f(n) > 0$ and $f(m)-f(m+1)\geq 1$. This process stops before we reach a multigraph $G(n')$ on $n' = n^{1/3}$ vertices, as then
$$
n-n'\leq \sum_{m=n'}^{n-1}(f(m)-f(m+1))< f(n')<(r-1)\binom{n'}{2}<(r-1)n'^2/2,
$$
contradicting to $n>r^3$. Therefore, without loss of generality, we may assume $\delta(G)\geq (r-1)(n-\lceil \frac{n}{r-1}\rceil) + r-1$. 

As $e(G) > (r-1) \cdot t(n,r-1)+(r-1)n$, there exists an edge, say $v_1v_2$, such that $w(v_1v_2) = r-1$, i.e., edge $v_1v_2$ presents in all $G_1, \cdots , G_{r-1}$. Let $T$ be a set of vertices with $|T|=t\leq r-1$. Then 
\begin{equation*}\label{eq:T}
\begin{aligned}
    e_G(T,V(G)\backslash T)&=\sum_{v\in T}d_G(v)-2e(G[T])\\
    &\geq t\cdot \left( (r-1)\bigl(n-\bigl\lceil \frac{n}{r-1}\bigr\rceil\bigr) + r-1 \right) -(r-1)t(t-1)\\
    &>(r-1)t\left(\frac{r-2}{r-1}n-t+1\right)\\
    &=(n-t)\cdot t(r-2) + t(r-1-t),
\end{aligned}
\end{equation*}
where the second inequality holds since $\lceil x \rceil <x +1$. 
Therefore, there exists a vertex $u \in V(G)\backslash T$ with $$e_G(u,T) > t(r-2) + \frac{t(r-1-t)}{n-t}.$$ 
Since $e_G(u,T)$ is an integer, we conclude that for any set $T$ of size $t \leq r-1$, there exists a vertex $u \in V(G)\backslash T$ with 
\begin{equation}\label{eq:cut-edges}
    e_G(u,T) \geq t(r-2)+1.
\end{equation}

Starting from $\{v_1,v_2\}$, iteratively apply the \eqref{eq:cut-edges} to obtain $v_3,v_4,\cdots, v_r$ such that for each $i\in \{3,\cdots,r\}$, 
\begin{align*}
    &e_G(v_i,\{v_1,\cdots,v_{i-1}\})=\sum\limits_{j=1}^{i-1} w(v_iv_j) \\
    &\geq (i-1)(r-2) + 1 = (i-2)(r-1) + r-i+1.
\end{align*}
Therefore, there are at least $r-i+1$ colors that appear simultaneously on all edges of the form $v_i v_j$ with $1 \le j \le i-1$, which we denote by the star $S_i$. 
Assign these common colors greedily so that $S_r, \dots, S_3$ receive different colors. 
Together with $w(v_1v_2)=r-1$, this yields an $(r-1)$-star edge-colored $K_r$, a contradiction. 
Therefore, our initial assumption was false, and $e(G) \le (r-1) \cdot t(n,r-1)+(r-1)n$. 
This completes the proof.
\end{proof}

\section{Counterexamples}\label{sec:counterexample}
In this section, we present the proofs of \cref{thm:counter} and Corollary \ref{cor:conterexample}, which exhibit an infinite family of counterexamples to Conjecture \ref{conj1}, demonstrating that the conjecture does not hold in general. 
We begin with a proposition as follows.
\begin{prop}\label{prop:q(F)}
    $q(F_{\mathcal{P}, t}) = t-1$.
\end{prop}
\begin{proof}
Let $(R,B)$ be a strong red–blue coloring of $F_{\mathcal{P},t}$, with red set $R$ and blue set $B$. 
By Construction~\ref{cons:conj2}, we have $V(F_{\mathcal{P},t})=A\cup [t]$ with $A=\{a_1,\dots,a_{t-1}\}$, and the edge set is exactly
$$\bigl\{\,a_i\,i\,j:\ i\in[t-1],\ j\in\{i+1,\dots,t\}\,\bigr\}.$$
We first note that
$$|R\cap [t]|\le 1,$$
since every pair $\{u,v\}\subseteq [t]$ lies in some edge $a_k u v$ (with $k=\min\{u,v\}$), whereas a strong coloring permits exactly one red vertex per edge. 

\noindent\textbf{Case 1: $R\cap [t]=\emptyset$.}
Then all vertices of $[t]$ are blue. For each $i\in[t-1]$ choose $j=i+1$. The edge $a_i\,i\,j$ has its two $[t]$-vertices blue, hence $a_i$ must be red. Thus $a_1,\cdots,a_{t-1}\in R$ and so $|R|=t-1$.

\noindent\textbf{Case 2.} Assume $R\cap[t]=\{r\}$ for some $r\in[t]$. We claim that $r=1$. Suppose that $r\ne 1$. Choose any $j\in\{2,\dots,t\}\setminus\{r\}$. 
Then $1,j\in B$, so the edge $a_1 1j\in E(F_{\mathcal{P},t})$ forces $a_1\in R$ to keep exactly one red per edge. 
But then the edge $a_11r\in E(F_{\mathcal{P},t})$ contains two red vertices, contradicting that $(R,B)$ is a strong red-blue coloring. 
Hence $r=1$. Consequently $[t]\setminus\{1\}\subseteq B$. For each $i\in\{2,\cdots,t-1\}$ take $j=i+1$; the edge $a_i ij\in E(F_{\mathcal{P},t})$ has $i,j\in B$, so it forces $a_i\in R$. 
Finally, for the edge $a_1 12$, we already have $1\in R$ and $2\in B$, hence $a_1\in B$ to ensure exactly one red. Hence $R=\{1,a_2,\cdots,a_{t-1}\}$ and $|R|=t-1$. 

Therefore, every strong red–blue coloring satisfies $|R|=t-1$, and hence $q(F_{\mathcal{P},t})=t-1.$
\end{proof}

\begin{proof}[Proof of \cref{thm:counter}]
    The lower bound in Theorem~\ref{thm:counter} comes from the following two constructions. 
Let $U$ and $V$ be disjoint vertex sets with $|U|=s$ and $|V|=n-s$. We place on $V$ an auxiliary Tur\'an graph $T(n-s,t-1)$ with vertex classes $V_1,\cdots,V_{t-1}$.

\medskip
\noindent\textbf{Construction A ($s\le t-2$).}
Consider the $3$-graph $H_A$ on $U\cup V$ with edge set
\[
E_A=\bigl\{\,u v w:\ u\in U,\ \{v,w\}\in\tbinom{V}{2}\,\bigr\}.
\]
Every edge meets $U$, so $\nu\!\left(H_A\right)\le |U|=s$. 
Moreover, by Proposition \ref{prop:q(F)}, $q(F_{\mathcal P,t}) = t-1$, hence $H_A$ is $F_{\mathcal P,t}$-free. 
Clearly, $|E_A|=s\binom{n-s}{2}$.

\medskip
\noindent\textbf{Construction B ($s\ge t-1$).}
    Consider the $3$-graph $H_B$ on $U\cup V$ with edge set
    $$E_B=\bigl\{u v w:\ u\in U,\ vw\in E(T(n-s,t-1))\,\bigr\}.$$
    Again every edge meets $U$, so $\nu\!\left(H_B\right)\le |U|=s$. 
    To see that $H_B$ is $F_{\mathcal P,t}$-free, note that the shadow of $F_{\mathcal P,t}$ contains a clique $K_t$ on the vertex set $B=[t]$ (see Construction~\ref{cons:conj2}). Indeed, for every $1\le i<j\le t$ the pair $\{i,j\}$ appears in the triple $a_iij$. 
    By contrast, the shadow of $E_B$ is the complete $t$-partite graph with parts $U,V_1,\dots,V_{t-1}$. Therefore, if $F_{\mathcal{P},t} \subseteq H_B$, then these $t$ vertices must be distributed among the sets $U, V_1, \dots, V_{t-1}$, with exactly one vertex in each set. 
    However, consider the $3$-edges $\{a_1 1 j : 2 \leq j \leq t\}$ in $F_{\mathcal{P},t}$. 
    Owing to the distribution of edges in $H_B$, these edges cannot occur simultaneously in $E_B$. Consequently, $H_B$ is $F_{\mathcal{P},t}$-free and 
    $$|E_B|=|U|\cdot t(n-s,t-1)=s\cdot t(n-s,t-1).$$
  
    It remains to prove the upper bound. 
    Let $H$ be an $F_{\mathcal P,t}$-free $3$-graph on the vertex set $[n]$ with $\nu(H)\le s$. 
    By applying Lemma \ref{fact} on $H$, we obtain a partition $A\cup B$ of $V(H)$ such that $|A| \leq s$ and $e(H[B]) \leq 9s^2n$. 

    Since $H$ is $F_{\mathcal{P},t}$-free, the family of link graphs $$L_H(a_1,B),\ L_H(a_2,B),\ \dots,\ L_H(a_{|A|},B) \qquad (A=\{a_1,\dots,a_{|A|}\})$$
    is $(t\!-\!1)$-star edge-colored $K_t$-free. Then by \cref{thm:star colored K_r-free}, we have 
    $$e(H[A,B]) \leq 
    \begin{cases}
        (n-|A|) \binom{|A|}{2} + |A|\binom{n-|A|}{2}, & \text{if} \ |A| \leq t-2,\\
        (n-|A|) \binom{|A|}{2} + |A|\cdot t(n-|A|, t-1) + |A|(n-|A|), & \text{if} \ |A| \geq t-1.
    \end{cases}
    $$ 
    Hence \begin{align*}
        e(H) &= e(H[A])  + e(H[B]) + e(H[A,B]) \notag \\
        &\leq \binom{|A|}{3} + 9s^2n + e(H[A,B]) \notag \\
        &\leq \begin{cases}
        s\binom{n-s}{2} + O(n), & \text{if } s \leq t-2,\\
        s\cdot t(n-s, t-1) + O(n), & \text{if } s \geq t-1.
        \end{cases}
    \end{align*}
  which completes the proof of \cref{thm:counter}. 
\end{proof}

\begin{proof}[Proof of Corollary \ref{cor:conterexample}]
According to Construction~\ref{cons:conj2}, for each $v\in V(F_{\mathcal P,t})$ the link graph $L_{F_{\mathcal P,t}}(v)$ is $2$-colorable, i.e., $\chi\bigl(L_{F_{\mathcal P,t}}(v)\bigr)=2$. By Proposition \ref{prop:q(F)}, $q(F_{\mathcal P,t}) = t-1$. Invoking \cref{cons:conj} together with Conjecture~\ref{conj1} (taking $F=F_{\mathcal P,t}$) would then give
$$\ex\bigl(n,\{F_{\mathcal P,t},M_{s+1}^3\}\bigr)
=\max_{i\le q(F_{\mathcal P,t})}\,|E(H_i)|+o(n^2).$$ But here $\ell_i=\chi(L_{F_{\mathcal P,t}}(i))=2$ for every $i$, so type (2) in \cref{cons:conj} contributes no hyperedges in $H_i$, and hence
$$|E(H_i)|=(i-1)\binom{n-s}{2},\quad\text{for all }i\le q(F_{\mathcal P,t}) = t-1.$$
Therefore
$$\ex\bigl(n,\{F_{\mathcal P,t},M_{s+1}^3\}\bigr)
\le \bigl(t-2\bigr)\binom{n-s}{2}+o(n^2).$$
However, by \cref{thm:counter}, we have for $s> t-1$,
$$\ex\bigl(n,\{F_{\mathcal P,t},M_{s+1}^3\}\bigr)
\geq s\cdot t(n-s,t-1)
= s\cdot \frac{t-2}{t-1}\binom{n-s}{2}+o(n^2),$$
which exceeds $\,(t-2)\binom{n-s}{2}+o(n^2)$. 
This contradicts Conjecture~\ref{conj1}.
\end{proof}

\section{Proof of Theorem \ref{thm:J_t}}\label{sec:examples}
In this section, we prove \Cref{thm:J_t}, which provides positive examples supporting Conjecture \ref{conj1}. 

\begin{proof}[Proof of \cref{thm:J_t}]
Since $\chi(J_t)=2$ and $q(J_t) = 1$,  the lower bounds for each $s,t$ are given by \cref{cons:conj}. The following constructions give the lower bounds for the `in particular' case. 
\begin{itemize}
    \item For $s = 2$ and odd $n$, define $H(n, s)$ as the $3$-graph obtained from the complete $3$-partite $3$-graph $H = (\{1,2\}, V_1, V_2)$ with $|V_1| = |V_2| = \frac{n-3}{2}$, by adding a vertex $3$ and an edge $123$ such that for each $i \in [2]$ and $u \in V_i$, $iu3$ forms an edge. 
    \item For $s = 1$, or $s=2$ and even $n$, define $H(n, s)$ be the complete $3$-partite $3$-graph on $n$ vertices with one partite set of size $s$, and the other two of sizes $\left\lfloor \frac{n - s}{2} \right\rfloor$ and $\left\lceil \frac{n - s}{2} \right\rceil$. 
\end{itemize}
    
    \noindent Clearly, $$e(H(n,s)) =
    \begin{cases}
    2\lfloor\frac{(n-2)^2}{4}\rfloor+1, &\text{ if } s=2 \text{ and odd } n,\\
    s\lfloor\frac{(n-s)^2}{4}\rfloor, &\text{ if } s=1, \ \text{or} \ s=2 \text{ and even } n.
    \end{cases}$$ 
    It is straightforward to verify that $\nu(H(n,s))\le s$ in all cases. Moreover, When $s=1$ or when $s=2$ and $n$ is even, the hypergraph $H(n,s)$ is clearly $K_4^{3-}$-free. When $s=2$ and  $n$ is odd, we observe that for every $x\in V(H(n,s))$, the link graph of $x$ in $H(n,s)$ is bipartite, which implies $H(n,s)$ is also $K_4^{3-}$-free. It remains to prove the upper bounds. 

    Let $H$ be the maximum $J_t$-free $3$-graph on $[n]$ with matching number at most $s$. Let 
    $$A:=\{v\in V(H) \colon d_{H}(v) \geq 3sn+1\} \text{ and } B:= [n] \setminus A.$$ 
    By \cref{fact}, we have $|A| \leq s$ and $e(H[B]) \leq 9s^2n$. Since $H$ is $J_t$-free, we observe that for any $a\in A$, the link graph $L_H(a,B)$ is $K_t$-free. Consequently, 
    \begin{equation}\label{eq:Jt}
     e(L_{H}(a, B))=d(a,B)\leq t(|B|, t-1) \text{ for any } a\in A. 
    \end{equation}  
    This yields
    \begin{align*}
        e(H) &= e(H[A]) + e(H[A,B]) +e(H[B]) \\ &\leq \binom{|A|}{3} + \binom{|A|}{2}\cdot |B| + |A|\cdot t(|B|,t-1) + 9s^2n \\
        &\le \binom{s}{3} + \binom{s}{2}\cdot (n-s) + s\cdot t(n-s,t-1) + 9s^2n \\
        &\le s\cdot t(n-s, t-1)+O(n),
    \end{align*}
    where the second inequality follows from the fact that $|A|\le s$ and $|B|=n-|A|\ge n-s$. 
    
    Now we consider the case $t=3$ and $s=1,2$. Suppose for a contradiction that $e(H) > e(H(n,s))$. 
    If $s=1$, then $|A|\le s=1$. If $|A|=0$, then it is easy to check that $e(H)=e(H[B])\le 9n$, a contradiction. If $|A|=1$, then by \cref{eq:Jt} and \cref{fact} (ii), it follows that $e(H)=e(H[A,B])\le t(|B|,t-1)=t(n-1,2)$, which contradicts the assumption that $e(H) > e(H(n,1))$. 
    Now consider the case $s=2$. If $|A|\le 1$, then the same argument as above implies $$e(H(n,2)) < e(H)=e(H[A,B])+e(H[B])\le t(n-1,2)+36n,$$ contradicting $n\geq 150$. 
    
    Without loss of generality, let $A=\{1,2\}$. Then 
    $$e(H(n,2)) < e(H)=e(H[A,B])\le d_H(1,B)+d_H(2,B)+d_H(1,2).$$
    By \cref{eq:Jt}, it follows that $d_H(1, 2)\geq 1$. Let $U=N_H(1,2)$ and $\overline{U}=B\setminus U$. Since $H$ is $K_4^{3-}$-free, we conclude that $d_{H}(1,U)=d_{H}(2,U)=0$, $E(L_H(1)[U, \overline{U}] \cap L_H(2)[U, \overline{U}])=\emptyset$  and both $L_H(1,\overline{U})$ and $L_H(2,\overline{U})$ are triangle-free. Therefore, 
    \begin{align*}
        d_H(1,B)+d_H(2,B)&=d_H(1,U)+d_H(1,\overline{U})+e(L_H(1)[U, \bar{U}])+d_H(2.U)+d_H(2,\overline{U})+e(L_H(2)[U, \overline{U}])\\
        &\le 2\cdot t(|\overline{U}|,2)+|U|\cdot |\overline{U}|, 
    \end{align*}
    which implies 
    \begin{align*}
        e(H)\leq 2\left\lfloor \frac{(n-2-|U|)^2}{4}\right\rfloor+|U|(n-2-|U|)+|U|=:f(|U|)
    \end{align*}
    \noindent Then, 
    $$e(H)\leq \max_{|U| \geq 1} f(|U|) = f(1) = n-2 + 2\lfloor \frac{(n-3)^2}{4}\rfloor = e(H(n,2)),$$
    a contradiction. This completes the proof  of \cref{thm:J_t}.   
\end{proof}

By definition, it is readily verified that $\chi(J_t)=2$ and $q(J_t) = 1$, which satisfy the condition of Conjecture~\ref{conj1}. Therefore, Conjecture \ref{conj1} states $\ex(n,\{J_t, M_{s+1}^3\}) = s\cdot t(n-s, t-1) + o(n^2)$. Combining this with \cref{thm:J_t}, which proves the matching upper bound. 

\begin{cor}
There exists an infinite family of $3$-graphs that satisfies Conjecture \ref{conj1}; in particular, the family $\{J_t:t\ge 3\}$ does.  
\end{cor}

\section{Proof of \cref{thm:f32}}\label{sec:f32}

In this section, we present the proof of \cref{thm:f32}. It suffices to show that 
$$\ex(n, \{F_{3,2}, M_{s+1}^3\}) \le s \binom{n - s}{2},$$  
and $\mathcal{H}(n,s)$ be the unique extremal hypergraph.

\begin{proof}[Proof of \cref{thm:f32}]
Let $H$ be an $F_{3,2}$-free $3$-graph on the vertex set $[n]$ that maximizes the number of edges among all such graphs whose matching number is at most $s$. Clearly, $e(H)\geq s\binom{n-s}{2}$. By applying Lemma \ref{fact}, we obtain a partition $A\cup B$ of $V(H)$ such that $|A| \leq s$ and $e(H[B]) \leq 9s^2n$. Our first step is to prove that the size of $A$ is $s$. 

\begin{claim}\label{clm:size A}
$|A|=s$.
\end{claim}
\begin{proof}
By Lemma \ref{fact}, we have  $|A|\leq s$ and
\begin{align*}
e(H) &= e(H[A])  + e(H[B]) + e(H[A,B]) \notag \\
&\leq \binom{|A|}{3} + 9s^2n +  |A| \binom{n-|A|}{2} + (n-|A|) \binom{|A|}{2}.
\end{align*}
A direct computation shows that if $|A|\leq s-1$, then the right-hand side is strictly less than $s\binom{n-s}{2}$ whenever $n\ge 12s^2$. Therefore, $|A|=s$.
\end{proof}

The following claim states that all edges in  $H$ belong to 
$H[A,B]$.

\begin{claim}\label{Aleqs}
Both $A$ and $B$ are weakly independent sets in $H$.
\end{claim}
\begin{proof}
By Lemma \ref{fact} (2) and Claim \ref{clm:size A}, we obtain that 
 $B$ is a weakly independent set in $H$. For $A$, suppose to the contrary that there exists an edge $a_1a_2a_3 \in H[A]$. Since $H$ is $F_{3,2}$-free, for each pair $\{b_i, b_j\} \in \binom{B}{2}$, we have 
 $$|N(b_i,b_j)\cap \{a_1,a_2,a_3\}| \leq 2.$$ 
 Thus, recall that $|A|=s$ and for $n\ge 12s^2$, we have that 
\begin{align*}
e(H)&= e(H[A]) + e(H[A,B]) \\
& \leq \binom{s}{3} + s\binom{n-s}{2} + \binom{s}{2}(n-s) - \binom{n-s}{2} \\
& < s\binom{n-s}{2},
\end{align*}
which contradicts to the lower bound of $e(H)$.
\end{proof}

The main goal in the rest of the proof is to show that the extremal hypergraph is $\mathcal{H}(n,s)$. Define
$$\Delta^{\star} :=\max\{|N(u,v)| \colon u,v \in V(H)\}.$$ 
It follows from Claim \ref{Aleqs} that 
\begin{align}\label{eq:upper bound}
    e(H) &= \frac{1}{2}\sum_{x\in A, y\in B}|N(x,y)| \leq\frac{1}{2}|A||B|\Delta^{\star} = \frac{s}{2} (n-s)\Delta^{\star}. 
\end{align} 
Therefore, to prove the upper bound of $\ex(n,\{F_{3,2},M_{s+1}^3\})$, it suffices to show that $\Delta^{\star} \leq n - s - 1$.

To this end, we further use the  idea of a variant of Zykov's symmetrization to examine the degree of vertices in $A$, which plays an important role in establishing the bound on $\Delta^{\star}$.  

\begin{claim}\label{degree in A}
For each $a \in A$, $d_{H}(a) \geq \binom{n-s}{2}-(n-s)$. 
\end{claim}
\begin{proof}
Suppose that there exists one vertex $a \in A$ with $d_{H}(a)< \binom{n-s}{2}-(n-s)$. Observe that the degree of any two vertices in $A$ differ by at most $n-s$. Indeed, otherwise we can delete the vertex with smaller degree and duplicate the other one, thereby strictly increasing the number of edges in $H$, a contradiction to its maximality. This implies that for each $v\in A\backslash \{a\}$, we have $d_{H}(v)< \binom{n-s}{2}$. By Claim \ref{Aleqs} we can calculate that
\begin{align*}
e(H)\leq \sum_{v\in A}d_{H}(v)  < \binom{n-s}{2}-(n-s)+ (s-1)\binom{n-s}{2}<s\binom{n-s}{2},
\end{align*}
a contradiction.
\end{proof}

\begin{claim}\label{Ind set}
If $I$ is a weakly independent set in $H$ with $|I \cap B| \geq \sqrt{5s^2n}$, then $I \subseteq B$.
\end{claim}
\begin{proof}
Suppose on the contrary that there exists a vertex $a \in I \cap A$. Since $I$ is weakly independent, the link graph $L_{H}(a)$ satisfies 
\begin{align*}
e(L_{H}(a)) &= e(L_{H}(a)[A,B]) + e(L_{H}(a)[B])  \\
&\leq s(n-s) + \binom{n-s}{2}-\binom{|I\cap B|}{2} \\
&\leq s(n-s) + \binom{n-s}{2}-\binom{\sqrt{5s^2n}}{2},
\end{align*}
Then for $n\geq 12s^2$ and $s\geq 1$, it follows that
$$e(L_{H}(a))<\binom{n-s}{2}-(n-s),$$
which contradicts Claim \ref{degree in A}.
\end{proof}

\begin{claim}\label{clm:n-s-1}
      $\Delta^{\star}\leq n-s-1$.
\end{claim}
\begin{proof}
Suppose for a contradiction that $\Delta^{\star}\geq n-s$. Assume that $x,y \in V(H)$ with 
$$|N(x,y)|= \Delta^{\star}\geq n-s> \sqrt{5s^2n} +s=\sqrt{5s^2n} + |A|\ (n\geq 12s^2). $$
This means that $|N(x,y)\cap B|\geq \sqrt{5s^2n}$. Because $H$ is $F_{3,2}$-free, 
 $N(x,y)$ is a weakly independent set in $H$. 
It follows from Claim \ref{Ind set} that $N(x,y)\subseteq B$. Hence $x,y \in A$ and $N(x,y)=B$. Since $H$ is $F_{3,2}$-free, we infer that  $L_{H}(x,B)$ and $ L_{H}(y,B)$ are $2$-star edge-colored $K_3$-free.  By \cref{thm:2-colored-triangle}, $$
d(x,B)+d(y,B)\leq 2\left\lfloor \frac{(n-s)^2}{4}\right\rfloor\leq \frac{(n-s)^2}{2}. 
$$
Thus
\begin{align*}
e(L_{H}(x)) + e(L_{H}(y))&=e(L_{H}(x)[A,B]) + e(L_{H}(y)[A,B])+e(L_{H}(x)[B]) + e(L_{H}(y)[B])\\
&\leq 2s(n-s) + \frac{(n-s)^2}{2}\\
&<2\binom{n-s}{2}-2(n-s),
\end{align*}
Then there exists $x^*\in \{x,y\}$ such that $e(L_H(x^*))<\binom{n-s}{2}-(n-s)$, which contradicts Claim \ref{degree in A}.
\end{proof}
Moreover, the \cref{eq:upper bound} holds if and only if $|N(x,y)| = n - s - 1$ for each $x\in A, y\in B$. When $|N(x,y)|=n-s-1$, it follows from Claim \ref{Ind set} that  $N(x,y)\subseteq B$ and hence $N(x,y)=B\setminus\{y\}$. Then, for every $x\in A$, $N(x)=\binom{B}{2}$. Therefore, the extremal case $e(H) = s \binom{n-s}{2}$ occurs only when $H$ is isomorphic to $\mathcal{H}(n,s)$ defined in \Cref{construction2}. This completes the proof of \cref{thm:f32}. 
\end{proof}

Lastly, we establish that the exact Tur\'an numbers for $3$-graphs $F$ satisfying $q(F) = \infty$ are not universally equal. 

\begin{prop}
  There exists an infinite family of 3-graphs $\mathcal{F}$  such that any 3-graph  $F\in \mathcal{F}$  satisfies $q(F)=\infty$ and  
  $$\ex(n,\{F,M_{s+1}^3\})>\ex(n,\{F_{3,2},M_{s+1}^3\})= s\binom{n-s}{2}.$$   
\end{prop}
\begin{proof}
For $t\ge 3$, define $J_t^+$ to be the $3$-graph obtained from $J_t$ by adding a single $3$-edge whose three vertices lie entirely in the $[t]$-side. Clearly, $q(J_t^+)=\infty$. Set $\mathcal{F}=\{J_t^+:t\ge 3\}$. 

Let $H$ be a $3$-graph obtained from $\mathcal{H}(n,s)$ defined in \cref{construction2} by inserting a $J_t^+$-free $3$-graph in $A$. Every edge of $H$ still meets $A$, so
$\nu(H)\le |A|=s.$
We claim that $H$ is $J_t^+$-free. Indeed, suppose a contradiction that $H$ contains a copy of $J_t^+$. Since for every $v\in B$, the link graph $L_H(v)$ is bipartite, the center vertex of $J_t^+$ must embed into $A$. Consequently, the $[t]$-side of $J_t^+$ must embed into $B$; but then the added $3$-edge of $J_t^+$ would lie entirely inside $B$, which is impossible because $H$ has no edges contained in $B$. Hence $H$ is $J_t^+$-free.
Therefore,  
$$\ex(n,\{F,M_{s+1}^3\}\geq e(H)=s\binom{n-s}{2} + \ex(n,J_t^+) > s\binom{n-s}{2}$$ 
for all $s\geq 3$. 
\end{proof}

\section{Concluding Remark}\label{sec:concluding remarks}
In Construction~\ref{cons:conj2}, taking a matching partition $\mathcal{M}$ of $K_{2t}$ instead of $\mathcal{P}$ and assuming $s>2t-1$ also yields counterexamples to Conjecture~\ref{conj1}. In this case one checks that $q\left(F_{\mathcal{M},2t}\right)=2t-1$. Consequently, Conjecture~\ref{conj1} would predict
$$\ex\bigl(n,\{F_{\mathcal{M},2t},M_{s+1}^3\}\bigr)=(2t-2)\binom{n-s}{2}+o(n^2).$$
On the other hand, the $3$-graph $H_B$ defined in Construction~B is $F_{\mathcal{M},2t}$-free, with matching number at most $s$, and its size is
$$s\cdot t(n-s,2t-1)
= s\cdot \frac{2t-2}{2t-1}\binom{n-s}{2}+o(n^2).$$
Since $s>2t-1$, we have $\,s\cdot \tfrac{2t-2}{2t-1}>(2t-2)$, and hence
$$\ex\bigl(n,\{F_{\mathcal{M},2t},M_{s+1}^3\}\bigr)>(2t-2)\binom{n-s}{2}+o(n^2),$$
which contradicts the prediction of Conjecture~\ref{conj1}. However, we do not determine the asymptotic value in this regime; in particular, the asymptotic Tur\'an number for $\{F_{\mathcal{M},2t},M_{s+1}^3\}$ with $s>2t-1$ remains open.

Note that \cref{thm:star colored K_r-free} plays a central role in the proof of \cref{thm:counter}. However, the current argument yields only asymptotic bounds. Motivated by this, we propose the following conjecture aiming at the exact value. 
\begin{conj}\label{conj:general}For integers $r$ and $s\ge r-1$, let $G_1,G_2,\cdots,G_s\subseteq \binom{[n]}{2}$ be $(r-1)$-star edge-colored $K_r$-free. Then $$\sum_{i=1}^{s}e(G_i)\le s \cdot t(n,r-1).$$ 
\end{conj}

\cref{thm:f32} is currently verified only for $n \geq 12s^2$. Thus, it is natural to ask whether the same conclusion remains valid for all $n\ge 3s$. Meanwhile, \cref{thm:J_t} just determines the asymptotic value for $\{J_t,M_{s+1}^3\}$, the exact Turán number for this family remains open in general. In addition, we construct a $\{J_t,M_{s+1}^3\}$-free hypergraph with more edges than the previous one. 
\begin{construction}
    Let $A$ and $B$ be two disjoint vertex sets with $|A| = s$ and $|B| = n - s$, and define a $3$-graph $H$ on the vertex set $A \cup B$. Place a copy of $T(s,t-2)$ on $A$. On $B$, place a copy of $T(n-s,t-1)$ with vertex partition $B=\bigcup_{i=1}^{t-1}V_i$. Then $H$ consists precisely the following two kinds of edges:
    \begin{itemize}
        \item[$(1)$] Edges containing one vertex from $A$ together with an edge of $T(n-s,t-1)$ lying in $B$. 
        \item[$(2)$] Edges containing an edge of $T(s,t-2)$ lying in $A$ together with one vertex from $B\setminus V_i$, where $V_i$ is a fixed part of size $\lfloor \frac{n-s}{t-1}\rfloor$. 
    \end{itemize}
\end{construction}
\noindent One can check that the link graph of each vertex is $(t-1)$-partite. Then $H$ is $J_{t}$-free and has matching number at most $s$ with size $e(H)=s\cdot t(n-s,t-1)+(n-s-\lfloor \frac{n-s}{t-1}\rfloor)\cdot t(s,t-2)$.

\section*{Acknowledgement}
The work in this article was carried out in the ECOPRO group at IBS. We are grateful to Professor Hong Liu for his guidance and support throughout the completion of this work.
Nannan Chen, Yuzhen Qi and Miao Liu was supported by China Scholarship Council and Institute for Basic Science, IBS-R029-C4. Caihong Yang was supported by National Key R\&D Program of China (Grant No. 2023YFA1010202), the Central Guidance on Local Science and Technology Development Fund of Fujian Province (Grant No. 2023L3003), the Institute for Basic Science (IBS-R029-C4). 

\bibliographystyle{abbrv}
\bibliography{matching}
\end{document}